\documentclass[12pt]{amsart}
\usepackage[]{hyperref}
\usepackage{amssymb,amsfonts,amsmath,amsopn,amstext,amscd,latexsym, amsthm, enumerate,mathrsfs}
\usepackage{lipsum}
\usepackage[margin=32mm]{geometry}
\usepackage{bbm}
\usepackage[all,cmtip]{xy}

\usepackage{enumerate}
\usepackage[shortlabels]{enumitem}

\newtheorem{theorem}{Theorem}[section]

\newtheorem{lemma}[theorem]{Lemma}

\newtheorem*{thmA}{Theorem~A}
\newtheorem*{thmB}{Theorem~B}
\newtheorem*{thmC}{Theorem C}
\newtheorem*{thmD}{Theorem D}

\theoremstyle{remark}

\DeclareMathOperator{\kernel}{ker}

\DeclareMathOperator{\Aut}{Aut}

\DeclareMathOperator{\Irr}{Irr}
\DeclareMathOperator{\IBr}{IBr}

\newcommand{\bfC}{{\mathbf C}}

\newcommand{\bfN}{{\mathbf N}}

\DeclareMathOperator{\Sym}{Sym}

\newcommand{\OO}{\mathbf{O}}
\newcommand{\Centralizer}{\mathbf{C}}

\numberwithin{equation}{section}

\newcommand{\Alt}{{\mathrm {Alt}}}
\newcommand{\Out}{{\mathrm {Out}}}

\begin{document}

\title[Brauer characters]{ Brauer characters and normal Sylow $p$-subgroups}

\author[H. P. Tong-Viet]{Hung P. Tong-Viet}
\address{Department of Mathematical Sciences, Binghamton University, Binghamton, NY 13902-6000, USA}
\email{tongviet@math.binghamton.edu}


\subjclass[2010]{Primary 20C20; Secondary 20C15, 20B15}

\date{\today}

\keywords{Brauer characters; $p$-parts of character degrees; normal Sylow $p$-subgroups}

\begin{abstract} In this paper, we study some variations of the well-known  It\^{o}-Michler theorem for $p$-Brauer characters using  various inequalities involving the $p$-Brauer character degrees of finite groups. Several new criteria for the existence of a normal Sylow $p$-subgroup of finite groups are obtained using the $p$-parts and $p'$-parts of the $p$-Brauer character degrees.

 \end{abstract}

\maketitle
\section{Introduction}

Throughout this paper, $G$ will be a finite group and $p$ will be a fixed prime. Let $\Irr (G)$ be the set of all complex  irreducible characters of $G$ and  let $\IBr (G)$ be the set of irreducible $p$-Brauer characters of $G$.  

In this paper, we are interested in studying the influence of $p$-Brauer character degrees on the structure of finite groups. We will relate the $p$-parts and the $p'$-parts of the Brauer character degrees to the embedding of the Sylow $p$-subgroups of the group. Several recent papers have been devoted to the study of the $p$-parts of Brauer characters and the normal structure of the groups. See, for example \cite{LNTT,NTT}.

The Brauer character degrees do not behave well  with respect to divisibility. In general, it is not even true that a Brauer character degree divides the order of the group. For example, if $G=\textrm{McL}$ and $p=2$, then $G$ has an irreducible $p$-Brauer character $\varphi\in\IBr(G)$ with $\varphi(1)=2^9\cdot 7$ while $|G|=2^7\cdot 3^6\cdot 5^3\cdot 7\cdot 11.$ So $\varphi(1)$ does not divide $|G|$, the order of $G$ and furthermore, the $2$-part of $\varphi(1)$ is larger than the $2$-part of the order of $G$. This cannot happen for $p$-solvable groups as by the Fong-Swan theorem \cite[Theorem 10.1]{Navarro}, for every $\varphi\in\IBr(G)$ there exists $\chi\in\Irr(G)$ such that $\chi^\circ=\varphi$, hence  $\varphi(1)$ always divides $|G|$ in this situation. Recall that $\chi^\circ$ is the reduction of $\chi$ to $G^\circ,$ the set of all $p$-regular elements of $G$. In general, for each $\varphi\in\IBr(G)$, there exists some $\chi\in\Irr(G)$ such that $\varphi$ is a constituent of $\chi^\circ$ and thus $\varphi(1)\leq \chi(1)$. Therefore, every Brauer character degree is bounded above by some ordinary character degree.

If $n\ge 1$ is an integer and $p$ is a prime, then we write $n_p$ for the largest power of $p$ dividing $n$.
The well-known It\^{o}-Michler theorem for Brauer character states that a prime $p$ does not divide the degree of any $p$-Brauer characters of $G$ if and only if $G$ has a normal Sylow $p$-subgroup (see \cite[Theorem 5.5]{Michler}). Notice that $\OO_p(G)$, the largest normal $p$-subgroup of $G$, is contained in the kernel of every irreducible $p$-Brauer character of $G$. Thus this theorem basically says that $\varphi(1)_p=1$ for all $\varphi\in\IBr(G)$ if and only if $|G:\kernel(\varphi)|_p=1$ for all $\varphi\in\IBr(G)$. So, we would like to know how large $\varphi(1)_p$ can be in comparison with $|G:\kernel(\varphi)|_p.$
We first consider $p$-solvable groups.
\begin{thmA}\label{th:p-solvable}
Let $p$ be a prime and let $G$ be a finite $p$-solvable group. Then $$\varphi(1)_p^2\leq |G:\kernel(\varphi)|_p$$ for all $\varphi\in\IBr(G)$ if and only if $G$ has a normal Sylow $p$-subgroup.
\end{thmA}

Thus if $G$ is $p$-solvable and is not $p$-closed, i.e., $G$ has no normal Sylow $p$-subgroup, then $ |G:\kernel(\varphi)|_p<\varphi(1)_p^2$ for some $\varphi\in\IBr(G)$.

As suspected,  Theorem A does not hold for arbitrary finite groups. For $p=2$, we can take $G=\textrm{M}_{22}$. We see that $\varphi(1)_p^7\leq |G:\kernel(\varphi)|_p=|G|_p$ for all $\varphi\in\IBr(G)$.
 For $p=3$, we can take $G=\Alt(7)$, the alternating group of degree $7$; we have $\varphi(1)_p^2\leq |G:\kernel(\varphi)|_p=|G|_p$ for all $\varphi\in\IBr(G)$. In both cases, $G$ has no normal Sylow $p$-subgroup.

Define \[\mu=\mu_p=\begin{cases}
      2,& \text{ if $p\ge 5$}\\
      3,& \text{ if $p=3$},\\
      9, &\text{ if $p=2$}.
\end{cases}\]
For general groups, we obtain  the following.
\begin{thmB} Let $p$ be a prime and let $G$ be a finite group. Then $$\varphi(1)^{\mu_p}_p\leq |G:\kernel(\varphi)|_p$$ for all $\varphi\in\IBr(G)$ if and only if $G$ has a normal Sylow $p$-subgroup.
\end{thmB}
This generalizes several  results obtained in \cite{CCLT,Gagola,GL,Qian} for both ordinary and Brauer characters.

\medskip

In \cite{CH}, it was shown that if $G$ has an ordinary character $\chi\in\Irr(G)$ such that $|G|/\chi(1)$ is a prime power, then $G$ cannot be a nonabelian simple group.  The number $|G|/\chi(1)$ is called the co-degree of $\chi$. Although other authors called $|G:\kernel(\chi)|/\chi(1)$ the co-degree of $\chi$. The co-degrees of ordinary characters have been studied by several authors in \cite{CH,CMM,Qian2}.
Except for $p$-solvable groups, it makes no sense to define the `co-degree' for $p$-Brauer characters. Here is what we could do to extend Chillag and Herzog's result to Brauer characters. The idea is to replace divisibility by inequality. Let $\chi\in\Irr(G)$ be an irreducible ordinary character of $G$. Suppose that $|G:\kernel(\chi)|/\chi(1)=r^m$ for some prime $r$ and integer $m\ge 1.$ Then $|G:\kernel(\chi)|=\chi(1)r^m$ and thus $|G:\kernel(\chi)|_{r'}=\chi(1)_{r'}$. In particular, we have $|G:\kernel{\chi}|_{r'}\leq \chi(1)_{r'}.$ 

Now using this inequality for Brauer characters, we can  extend Theorem 1 in \cite{CH}  to Brauer characters  as follows.

\begin{thmC}
Let $p$ be a prime and let $G$ be a finite group. If $$\varphi(1)_{r'}\ge |G:\kernel(\varphi)|_{r'}$$ for some prime $r$ and some $p$-Brauer character $\varphi\in\IBr(G)$, then $G/\kernel(\varphi)$ is not a nonabelian simple group.
\end{thmC}

If $G$ has a faithful character $\beta\in\IBr(G)$ and $\beta(1)_{r'}\ge |G|_{r'}$ for some prime $r$, then $G$ is not nonabelian simple by Theorem C. This gives a nice non-simplicity condition for finite group. Obviously, if $|G|$ is indivisible by $p$, then $\beta$ can be considered as an ordinary character of $G$ and thus Theorem C is a genuine generalization of Theorem 1 in \cite{CH}.

It follows from Theorem 2A in \cite{CH} that if $|G|/\chi(1)$ is a prime power for all $\chi\in\Irr(G)$ with $\chi(1)>1$, then $G$ has a normal Sylow $p$-subgroup for some prime $p$ and $G/P$ is abelian. Notice that the hypothesis implies that $G/\chi(1)$ must be a power of a fixed prime which is $p$ in this case.

We now suppose that for every $\varphi\in\IBr(G)$ with $\varphi(1)>1$, $\varphi(1)_{r'}\ge |G|_{r'}$ for some prime $r$. It is easy to see that there is a unique prime, say $s$, such that $\varphi(1)_{s'}\ge |G|_{s'}$ for all $\varphi\in\IBr(G)$ with $\varphi(1)>1$. Notice that if $p\neq s$, then  it is not true that  $G$ will have a normal Sylow $s$-subgroup. For example, take $G=\Sym(4)$, $p=3$ and $s=2.$ Then every nonlinear irreducible $3$-Brauer character $\varphi$ of $G$ is faithful and has degree $3$; also $|G|_{2'}=\varphi(1)_{2'}$ but $G$ has no normal Sylow $2$-subgroup.
\begin{thmD}
Let $p$ be a prime and let $G$ be a finite group. Then $$\varphi(1)_{p'}\ge |G:\kernel(\varphi)|_{p'}$$ for all $\varphi\in\IBr(G)$ with $\varphi(1)>1$ if and only if $G$ has a normal Sylow $p$-subgroup $P$ and $G/P$ is abelian.
\end{thmD}

The paper is organized as follows. In Section~\ref{sec:p-Part}, we will prove Theorems A and B. Theorems C and D will be proved in Section~\ref{sec:pprime-Part}.


\section{$p$-Parts of Brauer character degrees}\label{sec:p-Part}
We first collect some results which we will need for the proofs of our main theorems.
\begin{lemma}\label{lem:large orbit} Let $P$ be a nontrivial $p$-group that acts faithfully and coprimely on a group $H$. Then there exists an element $x\in H$ such that $|\bfC_P(x)|\leq (|P|/p)^{1/p}.$
\end{lemma}
\begin{proof} This is Theorem A in \cite{Isaacs}.
\end{proof}
\begin{lemma}\label{lem:coprime}
Let $P$ be a $p$-group and suppose that $P$ acts faithfully and coprimely on $N,$ where $N$ is a direct product of some nonabelian simple groups. Then $P$ has at least two regular orbits on $\Irr(N)$.
\end{lemma}
\begin{proof}
This is Corollary 2.8 in \cite{Qian}.
\end{proof}

\begin{lemma}\label{lem:p-parts} Let $S$ be any nonabelian simple group and let $p$ be a prime with $p$ dividing $|S|$. Then $|\Out(S)|_p<|S|_p$.
\end{lemma}

\begin{proof}
If the pair $(S,p)$ lies in the list $$\mathcal{L}= \{(\Alt(7),3),(\Alt(7),2),(\Alt({11}),2),(\Alt({13}),2),(\textrm{M}_{22},2)\},$$ then the lemma follows easily by checking \cite{Atlas}.

Now suppose that $(S,p)\not\in\mathcal{L}$. By \cite[Lemma 1.2]{Gagola}, there exists $\psi\in\Irr(S)$ such that $|\Aut(S)|_p<\psi(1)_p^2.$ Since $\psi(1)$ divides $|S|$, we deduce that $\psi(1)_p\leq |S|_p$ so $|\Aut(S)|_p<|S|_p^2$ which implies that $|\Out(S)|_p<|S|_p$ as wanted.
\end{proof}

\begin{lemma}\label{lem:defect zero}
Let $S$ be a nonabelian simple group and let $p$ be a prime. Then $S$ has a $p$-block of defect zero or one of the following cases holds:
\begin{itemize}
\item[$(i)$] $p=3$ and $S\cong \rm{Suz},\rm{Co}_3$ or $\Alt(n)$ for some integer $n\ge 7.$
\item[$(ii)$] $p=2$ and $S\cong \rm{M}_{12},\rm{M}_{22},\rm{M}_{24},\rm{J}_2,\rm{HS},\rm{Suz},\rm{Ru},\rm{Co}_1,\rm{Co}_3,\rm{B}$ or $\Alt(n)$ for some integer $n\ge 7.$
\end{itemize}
\end{lemma}
\begin{proof} This is Corollary 2 in \cite{GO}.
\end{proof}

Recall that if $B$ is a $p$-block of a finite group $G$ with defect $d=d(B)$ and assume $|G|_p=p^a$, then $p^{a-d}$ divides $\varphi(1)$ for all $\varphi\in\IBr(G)$ (see \cite[Corollary 3.17]{Navarro}). Furthermore, if $G$ has a block of defect zero, then $G$ has a character $\chi\in\Irr(G)$ such that $\chi^\circ\in\IBr(G)$ and $\chi(1)_p=|G|_p$ (\cite[Theorem 3.18]{Navarro}).

\begin{theorem}\label{th:pParts}
Let $S$ be a nonabelian simple group and let $p$ be a prime dividing $|S|$. Then one of the following holds.

\begin{enumerate}
\item[${(i)}$] There exists $\beta\in\IBr(S)$ such that $|\Aut(S)|_p<\beta(1)_p^2$.
\item[${(ii)}$] $p=3$ and $S\cong \Alt(7)$.
\item[${(iii)}$] $p=2$ and $S\cong \rm{M}_{22}$ or possibly $\Alt(n)$ with $n\in\{22,24,26\}$.
\end{enumerate}
\end{theorem}

\begin{proof}
Suppose first that $S$ has an irreducible character $\theta\in\Irr(S)$ of $p$-defect zero. Then $\beta=\theta^\circ\in\IBr_p(G)$ and $\beta(1)_p=\theta(1)_p=|S|_p$. Since $|\Out(S)|_p<|S|_p$, we deduce that $$\beta(1)_p^2=|S|_p^2>|S|_p\cdot |\Out(S)|_p=|\Aut(S)|_p.$$ So in this case, we are in situation $(i)$.

We assume from now on that $S$ has no  $p$-block of defect zero.
Then $S$ is isomorphic to one of the groups listed in  Lemma \ref{lem:defect zero}.
%

Write $|S|_p=p^{a}$. Then $|\Out(S)|\le 2$ and so $|\Out(S)|_p\leq p^c$ with $c=\delta_{2,p}$. Thus $|\Aut(S)|_p\leq p^{a+\delta_{2,p}}.$

If $B$ is a $p$-block of $S$ with defect $d=d(B)$, then $p^{a-d}$ divides $\beta(1)$ for all $\beta\in\IBr(B)$. Hence $\beta(1)_p^2\geq p^{2(a-d)}$ and thus we need to show that either  $(ii)$ or $(iii)$ of the theorem holds or $2(a-d)>a+\delta_{2,p}$ or equivalently
\begin{equation}\label{eqn1}
d< \frac{1}{2}(a-\delta_{2,p}).
\end{equation}

Now if Case $(ii)$ or $(iii)$ holds, then we are done. So, assume that the pair $(S,p)$ is not in those two cases.

Assume first that $p=3$. If $S\cong \textrm{Suz}$ or $ \textrm{Co}_3$, then $|\Aut(S)|_3=3^{7}$  and $S$ has a Brauer character $\beta\in\IBr(S)$ with $\beta(1)_3=3^{6}$.   In both cases, we see that $|\Aut(S)|_p<\beta(1)_p^2$.  If $S\cong\Alt(n)$, then it follows from \cite[Theorem 2]{CCLT} that $S$ has a $3$-block of defect $d$ such that $d\leq (a-1)/2<(a-\delta_{2,p})/2$, so \eqref{eqn1} holds.

Assume that $p=2$. If $S$ is a sporadic simple group which is isomorphic to neither $\textrm{Co}_1$ nor $\textrm{B}$, then by using \cite{GAP}, we can find a Brauer character $\beta\in\IBr(S)$ with $|\Aut(S)|_2<\beta(1)_2^2$. Assume that $S\cong \textrm{Co}_1$ or $\textrm{B}$. Then $\Aut(S)\cong S$ and $|S|_2=2^{41}$ and $2^{21}$, respectively. Moreover, $S$ has a $p$-block of defect $d=3$ in both cases. Thus $d< (a-1)/2$ and \eqref{eqn1} holds. Next, suppose that $S\cong \Alt(n)$ with $n\ge 5.$ It follows from \cite[Theorem 2]{CCLT} again that either $S$ has a block of defect $d$ with $d\leq (a-2)/2$ or $n\in \{7,9,11,13,22,24,26\}$. If the first case occurs, then  \eqref{eqn1} holds. Assume that the latter case occurs. Since we are not in Case $(iii)$, $n\in\{7,9,11,13\}$. For those values of $n$, we can find, using \cite{GAP}, a Brauer character $\beta\in\IBr(S)$ with $|\Aut(S)|_2<\beta(1)_p^2.$
The proof is now complete.
\end{proof}

We now give a proof of Theorem A which we restate here.
\begin{theorem}\label{th:theorem A}
Let $p$ be a prime and $G$ be a finite $p$-solvable group. Then $$\beta(1)_p^2\leq |G:\kernel(\beta)|_p$$ for all $\beta\in\IBr(G)$ if and only if  $G$ has a normal Sylow $p$-subgroup.
\end{theorem}

\begin{proof} If $G$ has a normal Sylow $p$-subgroup, then $p$ does not divide the degree of any $p$-Brauer irreducible character of $G$ by the It\^{o}-Michler theorem and the result follows.

 Now, suppose that 
 \begin{equation}\label{eqn}\beta(1)_p^{2}\leq |G:\kernel(\beta)|_p\text{ for all $\beta\in\IBr(G)$}.
 \end{equation}  We need to show that $G$ has a normal Sylow $p$-subgroup.
Let $G$ be a counterexample of minimal order. Since $\OO_p(G)$ is contained in the kernel of every $p$-Brauer character of $G$, we can assume that $\OO_p(G)=1.$ Let $P$ be a Sylow $p$-subgroup of $G$. Then $P$ is nontrivial.

Observe that if $M$ is any nontrivial normal subgroup of $G$, then $G/M$ also satisfies $\eqref{eqn}$ and since $|G/M|<|G|$, by the minimality of $|G|$, we have that $G/M$ has a normal Sylow $p$-subgroup, which is $PM/M$ and so $PM\unlhd G$. 

\medskip
(1) $G$ has a unique minimal normal subgroup $N$ and $PN\unlhd G$. Moreover, if $N$ is nonabelian, then $\Centralizer_G(N)=1$. 

Suppose that $N$ is a minimal normal subgroup of $G$. Then $PN\unlhd G$ by the observation above. Now, suppose that $G$ has a different normal subgroup, say $N_1\neq N.$ Then $PN_1\unlhd G.$ By Dedekind's  modular law, we have that $$PN\cap PN_1=P(N\cap PN_1)\unlhd G.$$ Notice that $N\cap PN_1\unlhd G.$ Since $N$ is a minimal normal subgroup of $G$, either $N\cap PN_1=1$ which implies that $P\unlhd G,$ a contradiction, or $N\cap PN_1=N$. This forces $N\leq PN_1$. Hence $NN_1\unlhd PN_1.$ Since $N\cap N_1=1$ and $|PN_1|=|P|\cdot |N_1|/|P\cap N_1|$, we deduce that $|NN_1|=|N|\cdot |N_1|$ divides $|PN_1|$. Thus $|N|$ divides $|P|$ which is impossible as $\OO_p(G)=1.$ Therefore, $G$ has a unique minimal normal subgroup $N$. Assume that  $N$ is nonabelian. Since $\Centralizer_G(N)$ is a normal subgroup of $G$ and does not contain $N$,  $N\cap \Centralizer_G(N)=1$. Now $\Centralizer_G(N)=1$ as otherwise $\Centralizer_G(N)$ would possess a minimal normal subgroup of $G$ which is different from $N$.
\medskip

(2) $P$ acts faithfully and coprimely on $N$. Since $N$ is a minimal normal subgroup of a $p$-solvable group $G$ with $\OO_p(G)=1$, we deduce that $p\nmid |N|$ so $P$ acts coprimely on $N$. We next claim that $P$ acts faithfully on $N.$ Since $PN\unlhd G$ by (1) and $P$ is a Sylow $p$-subgroup of $PN$, by Frattini's argument, $G=\bfN_G(P)N.$ Furthermore, $N$ centralizes and thus normalizes $\bfC_P(N)$ and also $\bfN_G(P)$ normalizes $\bfC_P(N)$, thus $\bfC_P(N)$ is normal in $G$. Since $\OO_p(G)=1$, we must have $\Centralizer_P(N)$=1.

\medskip
(3) $N$ is not abelian. Suppose by contradiction that $N$ is abelian. Then $N$ must be an elementary abelian $r$-group for some prime $r\neq p.$ 
By (2) $P$ acts coprimely and faithfully on  $N$ and thus it also acts coprimely and faithfully on $\Irr(N)=\IBr(N)$ as the actions of $P$ on $N$ and $\Irr(N)$ are permutation isomorphic.  Therefore, by Lemma \ref{lem:large orbit} there exists $\lambda\in\Irr(N)$ such that $|\bfC_P(\lambda)|\leq (|P|/p)^{1/p}$ which implies that $|P:\bfC_P(\lambda)|>\sqrt{|P|}$. Observe that $N\bfC_P(\lambda)$ is the stabilizer in $PN$ of $\lambda$ and thus there exists $\theta\in\IBr(PN)$ with $\theta(1)=\theta(1)_p>\sqrt{|P|}$. Let $\varphi\in\IBr(G)$ be an irreducible constituent of $\theta^G.$ Then $\varphi$ is faithful (since $N$ is the unique minimal normal subgroup of $G$) and $\varphi(1)_p\ge \theta(1)_p>\sqrt{|P|}.$
However, this means that $|G:\kernel(\varphi)|_p=|P|<\varphi(1)_p^2.$

(4) The final contradiction. We now know that $N$ is a minimal normal nonabelian subgroup of $G$ and  by (2), $P$ acts faithfully and coprimely on $N$. It follows from Lemma \ref{lem:coprime} that $P$ has a regular orbit on $\Irr(N)=\IBr(N)$. So, we can find $1\neq \lambda\in\IBr(N)$ and $\beta\in\IBr(G)$ which is a constituent of $\lambda^G$ such that $\beta(1)_p=|G|_p.$ Clearly $\beta$ is faithful and hence $\beta(1)_p^2=|G|_p^2>|G:\kernel(\beta)|_p$, which is a contradiction.

This contradiction proves the theorem.
\end{proof}


We now consider  the general case. The next result is Theorem B.

\begin{theorem}\label{th:p-parts}
Let $p$ be a prime and let $G$ be a finite group. Then 
\begin{equation} \label{eqn2}(\beta(1)_p)^{\mu_p}\leq |G:\kernel(\beta)|_p\end{equation}  for all $\beta\in\IBr_p(G)$  if and only if $G$ has a normal Sylow $p$-subgroup.
\end{theorem}

\begin{proof}
Notice that $\mu_p\ge 2.$ If $G$ has a normal Sylow $p$-subgroup, then $p$ divides the degree of no irreducible $p$-Brauer character of $G$ by \cite[Theorem 5.5]{Michler}, hence \eqref{eqn2} is vacuously true.

Conversely, assume that \eqref{eqn2} holds for all $\beta\in\IBr(G)$. We need to show that $G$ has a normal Sylow $p$-subgroup. Since $\OO_p(G)$ is contained in the kernel of every irreducible $p$-Brauer character of $G$, we can assume that $\OO_p(G)=1.$

Let $P$ be a Sylow $p$-subgroup of $G$ and let $N$ be any nontrivial normal subgroup of $G$. Then $G/N$ satisfies \eqref{eqn2} for all $\beta\in\IBr(G)$ and thus $G/N$ has a normal Sylow $p$-subgroup $PN/N$ by induction on $|G|$, hence $PN\unlhd G$. It follows that $G/N$ is $p$-solvable and thus if $N$ is $p$-solvable, then $G$ is $p$-solvable and the result follows from Theorem \ref{th:theorem A}.  Therefore, we can assume that every nontrivial minimal normal subgroup of $G$ is not $p$-solvable.

As in the proof of Claim (1) of Theorem \ref{th:theorem A}, $G$ has a unique minimal normal subgroup $N$, $PN\unlhd G$ and $\Centralizer_G(N)=1$. Thus $N\cong S^k$ for some nonabelian simple group $S$ with $p$ dividing $|S|$ and some integer $k\ge 1.$ 

Since $\Centralizer_G(N)=1$, $G$ embeds into $\Aut(N)\cong \Aut(S)\wr \Sym(k).$
Let $B=\Aut(S)^k\cap G.$ Then $G=BH$ where $H\cong G/B$ is a subgroup of $\Sym(k)$.

Write $|S|_p=p^{a}$ and $|\Out(S)|_p=p^{c}$. Let $\theta\in\IBr_p(S)$ be non-principal with $p\mid \theta(1)$ and let $\phi=\theta^k\in\IBr_p(N)$. Write $\theta(1)_p=p^{b}$ for some integer $b\ge 1$. Let $\varphi\in\IBr(G)$ lie over $\phi.$ Since $N$ is a unique  minimal normal subgroup of $G$ and $\theta\in\IBr(S)$ is faithful, $\phi=\theta^k$ is faithful, so is $\varphi.$  Since $G/N$ is $p$-solvable, $e:=\varphi(1)/\phi(1)$ divides $|G/N|$ by  Swan's Theorem \cite[Theorem 8.22]{Navarro}. By our hypothesis, we have $(\varphi(1)_p)^{\mu_p}\leq |G|_p$.  It follows that $$|H|_p\cdot |B/N|_p\cdot |N|_p\ge (\varphi(1)_p)^{\mu_p}=e_p^{\mu_p}(\phi(1)_p)^{\mu_p}=e_p^{\mu_p}(\theta(1)_p)^{k\mu_p}\ge p^{bk\mu_p}.$$ 
As $B/N\leq \Out(S)^k,$ we have \[|B/N|_p\leq |\Out(S)|_p^k\leq p^{ck}.\]
Furthermore, since $H\leq \Sym(k)$, we see that $|H|_p<p^{k/(p-1)}$ so 
\[p^{bk\mu_p}<p^{ak+k/(p-1)+ck}\leq p^{k(a+c+1)},\] which implies that $b\mu_p<a+c+{1}/{(p-1)}\le a+c+1$ and so
\begin{equation}\label{eqn3}
b\mu_p\leq  a+c.
\end{equation}

\medskip

We now apply Theorem \ref{th:pParts} for the nonabelian simple group $S$. We consider the following cases.

(a) $p=3$ and $S\cong\Alt(7)$.  In this case, we can take $\theta\in\IBr(S)$ with $\theta(1)=6$ so $b=1$. Furthermore, $a=2$ and $c=0$ so $a+c=2<3=b\mu_p$, contradicting \eqref{eqn3}.

(b)   $p=2$ and  $S\cong \textrm{M}_{22}$ or $\Alt(n)$ with $n=22,24$ or $26.$
Assume first that $S\cong  \textrm{M}_{22}.$ Then $S$ has a Brauer character $\theta\in\IBr(S)$ with $\theta(1)_2=2$ so $b=1$. Moreover, $a=7$ and $c=1$ so $a+c=8<9=b\mu_p$, contradicting \eqref{eqn3}.

Assume next that $S\cong\Alt(n)$ with $n\in\{22,24,26\}$. For these cases, we will use some $2$-Brauer character degrees of $\Alt(n)$ which can be found in Proposition 5.2 in \cite{BW}, for example.

Let $\lambda=(n-3,3)$ be a partition of $n$. Since $n\ge 22,$ we know that  $\lambda$ is a $2$-regular partition of $n$. It follows from \cite{Benson} that the simple module in characteristic $p$, $D^\lambda$, labelled by $\lambda$  remains irreducible upon restriction to $\Alt(n)$ and its dimension is given in \cite[Proposition 5.2]{BW}.  We have
$$\textrm{dim}D^\lambda=\begin{cases}
     \frac{1}{6}n(n-2)(n-7) & \text{ if $n\equiv 0$ mod $4$}, \\
       \frac{1}{6}n(n-1)(n-5) & \text{ if $n\equiv 1$ mod $4$}, \\
         \frac{1}{6}(n-1)(n-2)(n-6)& \text{ if $n\equiv 2$ mod $4$}, \\
           \frac{1}{6}(n+1)(n-1)(n-6) & \text{ if $n\equiv 3$ mod $4$}.
      
\end{cases}$$ 

If $n=22$ or $26$, then $S$ has a Brauer character $\theta\in\IBr(S)$ of degree $\theta(1)=(n-1)(n-2)(n-6)/6=1120$ or $2000$, so $b=5$ or $4$, respectively. In the first case, we see that $a+c=18+1=19<9\cdot 5=b\mu_p$ and in the second case $a+c=22+1<9\cdot 4=b\mu_p$. 

Assume next that $S\cong\Alt(24)$. Then $S$ has a Brauer character  $\theta\in\IBr(S)$ of degree $\theta(1)=24(24-2)(24-7)/6=1496$ and $b=3$. So $a+c=21+1=22<9\cdot 3=b\mu_p$, contradicting \eqref{eqn3}.

(c) For the remaining cases,  we have $\mu_p\ge 2$ and  there exists $\theta\in\IBr(S)$ with $a+c<2b\leq b\mu_p$ , where $p^b=\theta(1)_p$. However this contradicts \eqref{eqn3}.
The proof is now complete.
\end{proof}

\section{$p'$-Parts of Brauer character degrees}\label{sec:pprime-Part}
We first need the following result.
\begin{lemma}\label{lem:Sylow order}
If $S$ is a nonabelian simple group, then $|S|>|S|_r^2$ for all primes $r.$
\end{lemma}

\begin{proof}
We can assume that $r$ divides $|S|$, where $S$ is a finite nonabelian simple group. If $S$ has an $r$-block of defect zero then it has a character $\chi\in\Irr(S) $ with $\chi(1)_r=|S|_r$ and thus $|S|>\chi(1)^2\ge \chi(1)_r^2=|S|_r^2$ as wanted. So, we can assume that $S$ has no $r$-block of defect zero. By Lemma \ref{lem:defect zero}, we consider the following cases:

$(i)$ Case $r=3$. If $S$ is isomorphic to one of the sporadic simple groups in Lemma \ref{lem:defect zero} (i), then the lemma can be checked using \cite{GAP}. Now assume that $S\cong \Alt(n)$ for some $n\ge 7$. Firstly, we know that $|\Alt(n)|_3<3^{n/2}$ for all $n\ge 7.$ Now for the inequality $|S|>|S|_3^2$ it suffices to show  $n!>2\cdot3^{n}.$ This latter inequality can be proved using induction on $n\ge 7$.

$(ii)$ Case $r=2$. Again, if $S$ is one of the sporadic simple group in Lemma \ref{lem:defect zero}(ii) or $\Alt(7)$, then the result follows by checking \cite{GAP}. So, assume that $S\cong\Alt(n)$ with $n\ge 8.$ We have that $|S|_2=|\Alt(n)|_2<2^{n-1}$ and thus it suffices to show that $n!>2\cdot 2^{2n-2}=2^{2n-1}$. Now this inequality can be proved using induction on $n\ge 8$.
\end{proof}

We now give a proof of Theorem C. The proof is similar to that of Theorem 1 in \cite{CH}.
\begin{theorem}\label{th:co-degree} Let $p$ be a prime and let $G$ be a finite group. If $G$ has an irreducible  $p$-Brauer character $\varphi\in\IBr(G)$ such that $\varphi(1)_{r'}\ge |G:\kernel(\varphi)|_{r'}$ for some prime $r$, then $G/\kernel(\varphi)$ is not a nonabelian simple group.
\end{theorem}

\begin{proof}
Suppose by contradiction that $G$ has an irreducible  $p$-Brauer character $\varphi\in\IBr(G)$ such that $\varphi(1)_{r'}\ge |G:\kernel(\varphi)|_{r'}$ for some prime $r$ but $G/\kernel(\varphi)$ is a nonabelian simple group. Clearly, we can assume that $\kernel(\varphi)=1$ and thus $\varphi(1)_{r'}\ge |G|_{r'}$ where $G$ is a nonabelian simple group.

We see that $G$ has an ordinary character $\chi\in\Irr(G)$ such that $\chi^\circ$, the restriction of $\chi$ to the set of $p$-regular elements of $G$, contains $\varphi$ as an irreducible constituent. Thus $\chi(1)=\chi^\circ(1)\ge \varphi(1)\ge |G|_{r'}$. Thus $|G|_{r'}^2\leq \chi(1)^2<|G|$ which implies that $|G|_{r'}<|G|_r$ or $|G|<|G|_r^2.$ However, this would contradict Lemma \ref{lem:Sylow order}. The proof is now complete.
\end{proof}



We now finally give a proof of Theorem D.

\begin{theorem}
Let $p$ be a prime and $G$ be a finite group. Then $\varphi(1)_{p'}\ge |G:\kernel(\varphi)|_{p'}$ for all $\varphi\in\IBr(G)$ with $\varphi(1)>1$ if and only if $G$ has  a normal Sylow $p$-subgroup $P$ and $G/P$ is abelian.
\end{theorem}

\begin{proof} If $G$ has a normal Sylow $p$-subgroup $P$ such that $G/P$ is abelian, then $\IBr(G)=\Irr(G/P)$ contains only linear characters, so the `if' part of the theorem follows.

Let $P$ be a Sylow $p$-subgroup of $G$.
We now work to establish the `only if' part of the theorem. Suppose that 
\begin{equation}\label{eqn4} \varphi(1)_{p'}\ge |G:\kernel(\varphi)|_{p'} \text{ for all } \varphi\in\IBr(G) \text{ with } \varphi(1)>1.
\end{equation}

 If $P\unlhd G,$ then $\IBr(G)=\IBr(G/P)=\Irr(G/P)$ where $G/P$ is a $p'$-group.  Suppose by contradiction that $G/P$ is not abelian. Then we can find $\chi\in\Irr(G/P)$ with $\chi(1)>1$. By our hypothesis, we have $$\chi(1)_{p'}=\chi(1)\ge |G/P:\kernel(\chi)|_{p'}=|G/P:\kernel(\chi)|.$$ Notice that $\chi$ is a character of $(G/P)/\kernel(\chi)$ so the previous inequality occurs only when $\chi(1)=1$, a contradiction. Thus $G/P$ is abelian as wanted. Therefore, it suffices to show that $P\unlhd G.$

Suppose by contradiction that $P$ is not normal in $G$. Since $\OO_p(G)$ is contained in the kernel of every irreducible $p$-Brauer character of $G$, we can assume  $\OO_p(G)=1.$ 

\medskip
(1) $G$ has a unique minimal normal subgroup $N$ and $PN\unlhd G$. Moreover,  if $N$ is nonabelian, then $\Centralizer_G(N)=1$. Let $N$ be any minimal normal subgroup of $G$. Observe that every quotient of $G$ satisfies \eqref{eqn4}.  Thus $G/N$ has a normal Sylow $p$-subgroup $PN/N$ hence $PN\unlhd G$.
Suppose that $G$ has two distinct minimal normal subgroups, say $N_1$ and $N_2$. Then $PN_i\unlhd G$ for $i=1,2$. Argue as in the proof of Claim (1) of Theorem \ref{th:theorem A}, we get a contradiction. Thus $N$ is the unique minimal normal subgroup of $G$. It follows that if $N$ is nonabelian, then $\bfC_G(N)=1$. 

\medskip
(2) Since $G/N$ has a normal Sylow $p$-subgroup,  if $\varphi\in\IBr(G/N)$, then $p\nmid \varphi(1)$ by Theorem 5.5 in \cite{Michler}. Therefore, if $\beta\in\IBr(G)$ with $p\mid \beta(1)$, then $N$ is not contained in $\kernel(\beta)$ which implies that $\beta$ is faithful. Notice that such a character $\beta$ exists since $P$ is not normal in $G$. 

Now let $\beta\in\IBr(G)$ with $p\mid \beta(1)$. Then $\kernel(\beta)=1$ so $\beta(1)_{p'}\ge |G|_{p'}$ by \eqref{eqn4}, hence $\beta(1)\ge |G|_{p'}\beta(1)_p$. Now let $\chi\in\Irr(G)$ be such that $\beta$ is a constituent of $\chi^\circ,$ the restriction of $\chi$ to the $p$-regular elements of $G.$ Then $\chi(1)=\chi^\circ(1)\ge \beta(1)$ and $\chi(1)^2<|G|$ so $|G|>|G|_{p'}^2\beta(1)_p^2$ which implies that

\begin{equation}\label{eqn5}{|G|_p}/{\beta(1)_p^2}>|G|_{p'}.
\end{equation}

We now consider the following  cases.

(3) $N$ is a $p'$-group. As in the proof of Claim (2) of Theorem \ref{th:theorem A}, $P$ acts faithfully and coprimely on $N$. Now argue as in the proofs of (3) and (4) of Theorem \ref{th:theorem A} again, either $G$ has a faithful irreducible character $\varphi\in\IBr(G)$ with $\varphi(1)_p> |P|^{1/2}$ or $\varphi(1)_p=|P|$ depending on whether $N$ is abelian or not. In both cases, we have $\varphi(1)_p>|P|^{1/2}$ and thus $\varphi(1)_p^2|G|_{p'}>|G|_p|G|_{p'}=|G|\ge |G|_p$, contradicting \eqref{eqn5}.
So this case cannot happen.

\medskip

(4) $N\cong S^k$ for some integer $k\ge 1$ and some nonabelian simple group $S$ with $p$ dividing $|S|$. From (1), we know that $G$ embeds into $\Aut(N)=\Aut(S)\wr\Sym(k)$. Let $B=G\cap \Aut(S)^k$. Then there exists a transitive subgroup $H$ of $\Sym(k)$ such that $G=BH$ where $B\cap H=1$ and $N\unlhd B\unlhd G$ with $B/N\leq \Out(S)^k$.

Let $\theta\in\IBr(S)$ with $p\mid \theta(1)$. Such a character $\theta\in\IBr(S)$ exists since $p\mid |S|$ and $S$ is nonabelian simple. Now let $\lambda=\theta^k\in\IBr(N)$ and let $\varphi\in\IBr(G)$ be lying over $\lambda$. Then $\varphi(1)_p\ge \theta(1)_p^k.$

We now estimate the $p$-part of the order of $G$. We have that $|H|_p<p^{k/(p-1)}$ since $H\leq \Sym(k)$ and $|B|_p\leq |\Aut(S)|_p^k$ so $|G|_p= |B|_p\cdot |H|_p<|\Aut(S)|_p^k\cdot p^{k}$. Thus $$|G|_p/\varphi(1)_p^2<|\Aut(S)|_p^k\cdot p^k/\theta(1)_p^{2k}.$$
 
 Obviously, we have that $|G|_{p'}\geq |N|_{p'}=|S|_{p'}^k$ and thus to obtain a contradiction to \eqref{eqn5}, it suffices to show that $|\Aut(S)|_p^k\cdot p^k/\theta(1)_p^{2k}\leq |S|_{p'}^k$ or equivalently 
 
 \begin{equation}\label{eqn6}p|\Aut(S)|_p/\theta(1)_p^2\leq |S|_{p'}.\end{equation}
 
By Theorem \ref{th:pParts}, there exists $\beta\in\IBr(S)$ such that $|\Aut(S)|_p<\beta(1)_p^2$ or $p=3,S\cong \Alt(7)$ or $p=2$ and $S\in \{\textrm{M}_{22},\Alt(22),\Alt(24),\Alt(26)\}$.

$(i)$ Assume that $|\Aut(S)|_p<\beta(1)_p^2$ for some $\beta\in\IBr(S)$. Taking $\theta=\beta,$ we have $p|\Aut(S)|_p/\theta(1)_p^2<p.$ Moreover, we know from Lemma \ref{lem:Sylow order} that $|S|>|S|_p^2\ge p|S|_p$ so $|S|_{p'}>p$ and hence  \eqref{eqn6} holds in this case.

$(ii)$ Assume $p=3$ and $S=\Alt(7)$. In this case, there exists $\theta\in\IBr(S)$ with $\theta(1)_p=p$ and that $|\Aut(S)|_p=p^2.$ Therefore $p|\Aut(S)|_p/\theta(1)_p^2=p<|S|_{p'}=280.$ Thus \eqref{eqn6} holds in this case.

$(iii)$ Assume $p=2$ and $S\in \{\textrm{M}_{22},\Alt(22),\Alt(24),\Alt(26)\}$.

Assume first that $S\cong \textrm{M}_{22}$. Then $S$ has a character $\theta\in\IBr(S)$ with $\theta(1)_2=2$ and $|\Aut(S)|_2=2^8$. So  $p|\Aut(S)_p|/\theta(1)_p^2=2^7=128<|S|_{p'}=3465.$ 

Assume $S\cong\Alt(22)$. Then $|\Aut(S)|_2=2^{19}$ and $S$ has an irreducible $p$-Brauer character $\theta$ which is the restriction of $D^{(19,3)}$ to $S$ with $\theta(1)_2=2^5$. So  $$p|\Aut(S)|_p/\theta(1)_p^2=2^{10}<|\Alt(22)|_{2'}.$$ The remaining two cases can be argued similarly.
\end{proof}


\subsection*{Acknowledgment} The author is grateful to M.L Lewis, J.P. Cossey and X. Chen for  their help during the preparation of this paper. He  is also  grateful to  the referee for careful reading of the manuscript and for many helpful suggestions and corrections.


\end{document}